\documentclass[12pt]{article}
\usepackage{amsfonts}
\usepackage{amsmath}
\usepackage{amssymb}
\usepackage{latexsym}
\usepackage{amscd}
\usepackage{amsthm}
\usepackage{bbm}
\usepackage{url}
\usepackage{hyperref}

\date{}

\newtheorem{theorem}{Theorem}[section]
\newtheorem*{theorem*}{Theorem}
\newtheorem{lemma}[theorem]{Lemma}
\newtheorem{cor}[theorem]{Corollary}

\theoremstyle{definition}

\theoremstyle{plain}

\newcommand{\N}{\mathbb{N}}
\newcommand{\E}{\mathbb{E}}

\newcommand{\R}{\mathbb{R}}

\newcommand{\degreeset}{{\rm \bf Deg}}

\newcommand{\Event}{\mathcal{E}}
\newcommand{\A}{\mathcal{A}}

\def\Prob{{\mathbb P}}

\def\Col{{\rm col}}
\def\Row{{\rm row}}

\title{On the norm of a random jointly exchangeable matrix}
\author{Konstantin Tikhomirov and Pierre Youssef}

\begin{document}

\maketitle

\abstract{
In this note, we show that the norm of an $n\times n$ random jointly exchangeable matrix 
with zero diagonal can be estimated
in terms of the norm of its $n/2\times n/2$ submatrix located in the top right corner. 
As a consequence,
we prove a relation between the second largest singular values of a random matrix with constant row and column
sums and its top right $n/2\times n/2$ submatrix.
The result has an application to estimating the spectral gap of random undirected $d$-regular graphs in
terms of the second singular value of {\it directed} random graphs with predefined degree sequences.
}

\section{Introduction}

Given an $n\times n$ random matrix $M=(M_{ij})$ and a permutation $\sigma$ on $n$ elements,
we denote by $\sigma(M)$ the $n\times n$ matrix
$$(M_{\sigma(i)\sigma(j)})_{i,j=1}^n,$$
i.e.\ the $(i,j)$-th element of $\sigma(M)$ is equal to $M_{\sigma(i)\sigma(j)}$.
Further, we say that the matrix $M$ is \textit{jointly exchangeable} if
$M$ is equidistributed with $\sigma(M)$ for any non-random
permutation $\sigma$. Note that if $B$ is any random matrix and $\sigma$ is a random uniform permutation
independent of $B$ then $\sigma(B)$ is jointly exchangeable.

The jointly exchangeable matrices (arrays) have been previously studied in literature;
see, in particular, paper \cite{hoover} and \cite[Chapter~7]{kallenberg}.
Let us emphasize that the above definition is different from the notion
of a {\it separately exchangeable matrix} as well as a matrix with exchangeable entries.
In the former case, we assume that $M$ is equidistributed with $(M_{\sigma(i)\pi(j)})_{i,j=1}^n$
for any two permutations $\sigma$ and $\pi$, whereas the latter means that $M$, considered
as a sequence of $n^2$ elements, is exchangeable.
We refer to \cite{aldous,hoover} for a discussion of separately exchangeable arrays,
and to book \cite{kallenberg} for extensive information on the subject.

Limiting properties of the spectral distribution of random matrices with exchangeable entries
were considered, in particular, in \cite{chatterjee, ACW}.
In this note, we are interested in the problem of estimating the spectral norm of a {\it jointly exchangeable} matrix
in terms of the norm of its submatrix located in the top right corner.
Motivation for such a specific setting comes from a problem in the spectral theory of random graphs,
and can be seen in a more general context as a way to ``de-symmetrize'' a random matrix.
We will return to this in the second part of the introduction.

Let us note that the problem has a trivial solution if instead of joint exchangeability we consider
a {\it separately exchangeable} matrix.
Namely, if $M$ is separately exchangeable then it is easy to see that, 
denoting by $M^{(k\ell)}$ ($1\leq k,\ell\leq 2$) its block decomposition into four $n/2\times n/2$
submatrices, all the blocks $M^{(k\ell)}$ are equidistributed.
Hence, in view of the triangle inequality,
$$\Prob\{\|M\|\geq t\}\leq \Prob\{\|M^{(11)}\|+\|M^{(12)}\|+\|M^{(21)}\|+\|M^{(22)}\|\geq t\}\leq 4\Prob\{\|M^{(12)}\|\geq t/4\}$$
for any $t\geq 0$.
In the jointly exchangeable model, however, the principal submatrices $M^{(11)}$
and $M^{(22)}$ are generally not equidistributed with $M^{(12)}$, $M^{(21)}$,
and the above argument fails.

We address the problem in the following theorem:

\begin{theorem}\label{th: main}
There exists a universal constant $c_{\text{\tiny{\ref{th: main}}}}\in(0,1)$ with the following property. 
Let $n\geq 8$; let $M$ be a random $n\times n$ matrix with zero diagonal
and let $\sigma$ be the uniform random permutation of $\{1,2,\ldots, n\}$ independent of $M$. 
Further, let $\Event$ be an event such that $\Prob(\Event\,|\,M)\geq 1-c_{\text{\tiny{\ref{th: main}}}}$ everywhere on the probability space. 
Then, denoting by $T$ the $\lfloor n/2\rfloor\times\lfloor n/2\rfloor$ top right corner of $\sigma(M)$,
we have
$$\Prob\bigl\{\|M\|\geq \tau\bigr\}
\leq \frac{1}{c_{\text{\tiny{\ref{th: main}}}}}\Prob\bigl\{\|T\|\geq c_{\text{\tiny{\ref{th: main}}}}\tau\mbox{ AND }\Event\mbox{ holds}\bigr\}$$
for all $\tau>0$.
\end{theorem}

Let us note that we did not attempt to evaluate the constant $c_{\text{\tiny{\ref{th: main}}}}$.
Event $\Event$ in the above theorem provides additional flexibility,
which will be important in applications. In the particular case when $\Event$ coincides with the entire probability
space, we obtain the following
\begin{cor}
Let $n\geq 8$, let $M$ be an $n\times n$ jointly eachangeable random matrix with zero diagonal
and let $T$ be the $\lfloor n/2\rfloor\times\lfloor n/2\rfloor$ top right corner of $M$.
Then
$$\Prob\{\|M\|\geq \tau\}
\leq {c_{\text{\tiny{\ref{th: main}}}}}^{-1}\Prob\{\|T\|\geq c_{\text{\tiny{\ref{th: main}}}}\tau\},\quad \tau>0.$$
\end{cor}

It turns out that, under some extra assumptions, the above result can be
turned into a relation between second largest singular values.
In what follows, given an $n\times n$ matrix $M$, we denote by $s_1(M)\geq s_2(M)\geq\dots\geq s_n(M)$
its singular values arranged in non-increasing order (counting multiplicities).

Given an $n\times n$ matrix $A$ with nonnegative entries,
let us define $n$-dimensional vectors $u(A)$ and $v(A)$ via their coordinates as follows:
$$
u_i(A)=\Vert \Col_i(A)\Vert_1,\quad v_i(A)=\Vert \Row_i(A)\Vert_1,\quad i=1,\ldots,n,
$$
where $\Col_i(A)$ and $\Row_i(A)$ denote the $i$-th column and $i$-th row of $A$, respectively.
Now, given $u, v\in \R_+^n$ with $\Vert u\Vert_1=\Vert v\Vert_1$, 
we define $\A_n(u,v)$ as the set of all (non-random)
$n\times n$ matrices $A$ with nonnegative entries satisfying $u(A)=u$ and $v(A)=v$. 
When $u=v=d\, {\bf 1} =d(1,\ldots,1)$ for some $d>0$, we will use a shorter notation $\A_n(d)$
for the corresponding set.

It is easy to see that for any matrix $A$ from $\A_n(d)$,
$d$ is equal to its largest singular value $s_1(A)$, and the corresponding singular vector is ${\bf 1}$
(this can be checked, in particular, by considering the matrix $A^tA$ which has constant row and column sums equal to $d^2$).
Using the Courant--Fischer formula, or the singular value decomposition, we can express {\it the second largest} singular value of $A$ as
\begin{equation}\label{eq: s2-courant-fischer}
s_2(A)=\big\Vert A-\frac{d}{n}\, {\bf 1}\cdot{\bf 1}^t\big\Vert. 
\end{equation}
In the situation when $A$ is a random jointly exchangeable matrix with values in $\A_n(d)$,
formula \eqref{eq: s2-courant-fischer} enables us to use Theorem~\ref{th: main}
to estimate $s_2(A)$.
Generally, by passing to a submatrix of $A$ we destroy the double stochastic
structure. However, as we show below, with a high probability the submatrix will have ``almost constant''
row and column sums.
Given $m\in\N$ and two positive numbers $d$ and $\delta$, define 
\begin{align*}
\degreeset_m(d,\delta):=\Big\{&(u,v)\in \R_+^m\times \R_+^m:\, \Vert u\Vert_1=\Vert v\Vert_1\;\;\mbox{ AND }\\
&\big|\big\{i\leq m:\,\big|u_i(T)-d\big|> k\delta\big\}\big|
\leq m e^{-k^2}\mbox{ for all }k\in\N\;\;\mbox{ AND}\\
&\big|\big\{i\leq m:\,\big|v_i(T)-d\big|> k\delta\big\}\big|
\leq m e^{-k^2}\mbox{ for all }k\in\N\Big\}.
\end{align*}
Loosely speaking, when $\delta$ is small compared to $d$ and $(u,v)\in \degreeset_m(d,\delta)$,
matrices from $\A_m(u,v)$ are ``almost'' in $\A_m(d)$ 
in the sense that their row and column sums are close to $d$. 
Let us remark that the definition of $\degreeset_m(d,\delta)$ can be equivalently restated using the Orlicz norm
in $\R^n$ with the Orlicz function $\exp(t^2)$.
We are now ready to state the second main result of this note. 

\begin{theorem}\label{th: main2}
There exist positive universal constants $c,C$ such that the following holds. 
Let $n\geq C$ and let $d,\delta>0$ satisfy $d/\sqrt{\ln n}\geq C\delta>0$.
Let $A\in \A_n(d)$ be an $n\times n$ jointly exchangeable random matrix satisfying
$$C\|\Row_i(A)\|_2,\,C\|\Col_i(A)\|_2\leq \delta\;\;\;\mbox{a.s.,}\quad i\leq n.$$
Further, let $T$ be the $\lfloor n/2\rfloor\times\lfloor n/2\rfloor$ top right corner of $A$. 
Then for any $L\geq C$ we have 
\begin{align*}
\Prob\big\{s_2(A)\geq L\delta\big\}\leq \frac{1}{c}\Prob\Big\{ s_2(T)\geq cL\delta \mbox{ \ AND }
\big(u(T),v(T)\big)\in \degreeset_{\lfloor n/2\rfloor}\big(d/2,\delta\big)\Big\}.
\end{align*}
\end{theorem}

The above statement is used in an upcoming paper by the same authors \cite{TY}.
In fact, applying Theorem~\ref{th: main2} in \cite{TY}, we show that the spectral gap
of random $d$-regular undirected graphs in the uniform model can be bounded in terms of
the second singular value of ``almost $d$-regular'' random {\it directed} graphs.

The note is organized as follows:
In Section~2 we set the notation and
provide several auxiliary statements.
In Section~3 we prove Theorem~\ref{th: main} and complete
the paper with Section~4 with a proof of Theorem~\ref{th: main2}.
Let us emphasize that, although the primary application of all the results from this note
consists in ``de-symmetrization'' of $d$-regular random graphs, we think that
the simple arguments given here may turn out useful in other contexts, which
has been the main reason for grouping them in the separate paper.

\section{Preliminaries}
 
Given a vector $x\in\R^n$, we denote by $x_1,x_2,\dots,x_n$ its coordinates.
Further, by $\|x\|_p$ we denote the canonical $\ell_p^n$-norms ($1\leq p\leq \infty$).
Given an $n\times n$ matrix $M$, $\|M\|$ stands for the spectral norm of $M$.
A vector of ones will be denoted by ${\bf 1}$.
Given a natural number $m$, the set $\{1,2,\dots,m\}$ will sometimes be denoted by $[m]$.
For a real number $r$, $\lfloor r\rfloor$ is the largest integer not exceeding $r$ and $\lceil r\rceil$
is the smallest integer greater or equal to $r$. Given a finite set $I$, by $|I|$ we denote its cardinality.
Universal constants are denoted by $C,c$, etc.

Next, let us state (and prove) several auxiliary lemmas.
The following anti-concen\-tration estimate for a sum over a
uniform random subset of indices will be used in the proof of the main theorem. 
Its proof is quite standard, and we include it just for completeness.
\begin{lemma}\label{l: subsum}
Let $m\in \N$ ($m\geq 8$) and $1\leq k\leq \lceil m/2\rceil$. Further, let $a:=(a_i)_{i=1}^m$
be a sequence of reals, and $\mathcal S$ be a uniformly distributed random subset of $[m]$ of cardinality $k$.
Then
$$\Prob\Bigl\{\Bigl|\sum_{i\in\mathcal S}a_i\Bigr|\geq \frac{c_{\text{\tiny{\ref{l: subsum}}}}k}{m}\Bigl|\sum_{i=1}^m a_i\Bigr|\Bigr\}
\geq \frac{c_{\text{\tiny{\ref{l: subsum}}}}k}{m}$$
for a sufficiently small universal constant $c_{\text{\tiny{\ref{l: subsum}}}}>0$.
\end{lemma}
\begin{proof}
For each $i\leq m$, let $\chi_i$ be the indicator variable of the event $\{i\in\mathcal S\}$, so that
$\sum_{i\in\mathcal S}a_i=\sum_{i=1}^m a_i\chi_i$.
We have
\begin{align}\label{eq-second-moment-symetrization}
\E\Bigl(\sum_{i=1}^m a_i\chi_i\Bigr)^2&=\sum_{i\neq j}a_ia_j\Prob\{i,j\in\mathcal S\}+
\sum_{i=1}^m a_i^2\Prob\{i\in\mathcal S\}\nonumber\\
&=\frac{k(k-1)}{m(m-1)}\biggl(\sum_{i=1}^m a_i\biggr)^2+\frac{k}{m}\Bigl(1-\frac{k-1}{m-1}\Bigr)
\sum_{i=1}^m a_i^2,
\end{align}
where the last equality is due to the fact that $\Prob\{1,2\in\mathcal S\}=\frac{k(k-1)}{m(m-1)}$.
Next, denoting
$$t_u:=\E\prod_{v=1}^u\chi_v=\Prob\bigl\{[u]\subset\mathcal S\bigr\},\;\;\;u=1,2,3,4,$$
we can compute the fourth moment of the sum as follows:
\begin{align*}
\E\Bigl(&\sum_{i=1}^m a_i\chi_i\Bigr)^4\\
&=t_4\sum_{i\neq j\neq \ell\neq r}a_ia_ja_\ell a_r
+6t_3\sum_{i\neq j\neq \ell} a_i^2a_ja_\ell
+6t_2\sum_{i\neq j} a_i^2a_j^2
+4t_2\sum_{i\neq j} a_i^3a_j
+t_1\sum_{i=1}^m a_i^4\\
&=t_4\biggl(\sum_{i=1}^m a_i\biggr)^4
+6(t_3-t_4)\sum_{i\neq j\neq \ell} a_i^2a_ja_\ell
+6(t_2-t_4)\sum_{i\neq j} a_i^2a_j^2\\
&\hspace{2.7cm}+4(t_2-t_4)\sum_{i\neq j} a_i^3a_j
+(t_1-t_4)\sum_{i=1}^m a_i^4.
\end{align*}
Simplifying, we obtain
\begin{align*}
\E\Bigl(\sum_{i=1}^m a_i\chi_i\Bigr)^4
&=t_4\biggl(\sum_{i=1}^m a_i\biggr)^4
+6(t_3-t_4)\Bigl(\sum_{i=1}^m a_i^2\Bigr)\Bigl(\sum_{i=1}^m a_i\Bigr)^2
+6(t_2-t_3)\sum_{i\neq j} a_i^2a_j^2\\
&\hspace{1cm}+(4t_2+8t_4-12t_3)\Bigl(\sum_{i=1}^m a_i^3\Bigr)\sum_{i=1}^m a_i
+(t_1-4t_2+6t_3-3t_4)\sum_{i=1}^m a_i^4,
\end{align*}
whence, in view of the straightforward relations $t_u\leq \frac{k^u}{m^u}$,
\begin{align}\label{eq-4th-moment}
\E\Bigl(\sum_{i=1}^m a_i\chi_i\Bigr)^4
\leq
\frac{k^4}{m^4}\biggl(\sum_{i=1}^m a_i\biggr)^4
&+\frac{6k^3}{m^3}\Bigl(\sum_{i=1}^m a_i^2\Bigr)\Bigl(\sum_{i=1}^m a_i\Bigr)^2\nonumber\\
&+\frac{7k}{m}\Bigl(\sum_{i=1}^m a_i^2\Bigr)^2
+\frac{12k^2}{m^2}\Bigl(\sum_{i=1}^m a_i^2\Bigr)^{3/2}\Bigl|\sum_{i=1}^m a_i\Bigr|.
\end{align}
Finally, we consider two cases.
\begin{enumerate}
\item Assume that $\bigl|\sum\limits_{i=1}^m a_i\bigr|\geq 4\sqrt{\frac{m}{k}}\|a\|_2$.
Let $\eta:=\sum_{i\in\mathcal S}a_i-\frac{k}{m}\sum_{i=1}^m a_i$. 
From \eqref{eq-second-moment-symetrization}, we clearly have 
$$
\E \left(\sum_{i\in\mathcal S}a_i\right)^2\leq \frac{k^2}{m^2}\, \left(\sum_{i=1}^m a_i\right)^2 + \frac{k}{m} \sum_{i=1}^m a_i^2.
$$
Thus,  
$$
\E\eta^2=\E \left(\sum_{i\in\mathcal S}a_i\right)^2- \frac{k^2}{m^2}\, \left(\sum_{i=1}^m a_i\right)^2 \leq \frac{k}{m}\sum_{i=1}^m a_i^2.
$$
Using Markov's inequality together with the relation $\bigl|\sum\limits_{i=1}^m a_i\bigr|\geq 4\sqrt{\frac{m}{k}}\|a\|_2$, we get 
$$\Prob\Bigl\{|\eta|>\frac{k}{2m}\Bigl|\sum\limits_{i=1}^m a_i\Bigr|\Bigr\}
\leq\Prob\Bigl\{|\eta|>2\sqrt{\frac{k}{m}}\|a\|_2\Bigr\}\leq\frac{1}{4},$$
whence
$$\Prob\Bigl\{\Bigl|\sum_{i\in \mathcal S}a_i\Bigr|\geq\frac{k}{2m}\Bigl|\sum_{i=1}^ma_i\Bigr|\Bigr\}\geq \frac{3}{4}.$$

\item Assume that $\bigl|\sum\limits_{i=1}^m a_i\bigr|< 4\sqrt{\frac{m}{k}}\|a\|_2$. 
From \eqref{eq-second-moment-symetrization}, it is easy to see that 
$$
\E \left(\sum_{i\in\mathcal S}a_i\right)^2\geq  \frac{k^2}{2m^2}\, \left(\sum_{i=1}^m a_i\right)^2.
$$
Therefore we can write 
$$
\Prob\Bigl\{\Bigl|\sum_{i\in\mathcal S}a_i\Bigr|\geq \frac{k}{2m}\Bigl|\sum_{i=1}^m a_i\Bigr|\Bigr\}
\geq \Prob\Bigl\{\Bigl(\sum_{i\in\mathcal S}a_i\Bigr)^2\geq \frac{1}{2}\E\Bigl(\sum_{i\in\mathcal S}a_i\Bigr)^2\Bigr\}.
$$
Applying the Paley--Zygmund inequality, we deduce 
\begin{equation}\label{eq: paley-zygmund}
\Prob\Bigl\{\Bigl|\sum_{i\in\mathcal S}a_i\Bigr|\geq \frac{k}{2m}\Bigl|\sum_{i=1}^m a_i\Bigr|\Bigr\}
\geq \frac{1}{4}\biggl(\E\Bigl(\sum_{i\in\mathcal S}a_i\Bigr)^2\biggr)^2\biggl(\E\Bigl(\sum_{i\in\mathcal S}a_i\Bigr)^4\biggr)^{-1}.
\end{equation}
Using that $\bigl|\sum\limits_{i=1}^m a_i\bigr|< 4\sqrt{\frac{m}{k}}\|a\|_2$ together with \eqref{eq-4th-moment}, we get
$$
\E\Bigl(\sum_{i=1}^m a_i\chi_i\Bigr)^4\leq C\, \frac{k}{m}\Vert a\Vert_2^4,
$$
for some large positive constant $C$. Moreover, in view of \eqref{eq-second-moment-symetrization}, we  have 
$$
\E \bigg(\sum_{i\in\mathcal S}a_i\bigg)^2\geq \frac{k}{3m} \Vert a\Vert_2^2.
$$
Plugging the last two estimates into \eqref{eq: paley-zygmund}, we get the result. 
\end{enumerate}
\end{proof}

The next statement is, in a sense, converse to the last one as it establishes concentration
of the random sum around its mean.

\begin{lemma}\label{l: subsum2}
Let $m\in \N$ and $1\leq k\leq \lceil m/2\rceil$. Further, let $a:=(a_i)_{i=1}^m$
be a sequence of reals, and $\mathcal S$ be a uniformly distributed random subset of $[m]$ of cardinality $k$.
Then for any $t\geq 0$, we have  
$$
\Prob\Big\{ \Big\vert\sum_{i\in \mathcal{S}} a_i - \frac{k}{m}\sum_{i=1}^m a_i \Big\vert\geq t\Big\}
\leq  2\exp\left(-\frac{2t^2}{\Vert a\Vert_2^2}\right).
$$
\end{lemma}
\begin{proof}
Fix a $k\leq \lceil\frac{m}{2}\rceil$.
For each $i\leq m$, let $\chi_i$ be the indicator variable of the event $\{i\in\mathcal S\}$ so that
$\sum_{i\in\mathcal S}a_i=\sum_{i=1}^m a_i\chi_i$. 
It is known that $(\chi_i)_{i\leq n}$ are negatively associated (see for example \cite{NA} or \cite{Permantle}). 
Note that $\E\chi_i= \frac{k}{m}$ and that the random variable $a_i\chi_i$ lies in an interval of length
$\vert a_i\vert$ for any $i\leq m$. Hence, applying
Hoeffding's inequality for negatively associated random variables \cite{roussas}, we get the result.
\end{proof}
Note that a Bennett--type inequality can be derived in a similar manner. However, Lemma~\ref{l: subsum2}
is sufficient for our purposes.

\medskip

Finally, the following linear algebraic statement will be useful for us.
For the proof, see, for example, \cite[Corollary 8.1.30]{horn}.
\begin{lemma}\label{l: eigenvalue of positive vector}
Let $M$ be an $n\times n$ matrix with non-negative entries and let $\rho$ be its spectral radius.
Assume that $x$ is an eigenvector of $M$ and all coordinates of $x$ are strictly positive.
Then $Mx=\rho x$, i.e.\ the eigenvalue associated to $x$ is equal to the spectral radius.
\end{lemma}

\section{Proof of Theorem~\ref{th: main}}

The next lemma is the key to prove the main theorem. 

\begin{lemma}\label{l: submatrix}
There is a universal constant $c_{\text{\tiny{\ref{l: submatrix}}}}>0$ with the following property:
Let $n\geq 8$, $M$ be an $n\times n$ non-random matrix with zero diagonal, and let $\sigma$ be a uniform
random permutation on $[n]$. Denote by $T$ the $\lfloor n/2\rfloor\times \lfloor n/2\rfloor$ submatrix of $\sigma(M)$
located in its top right corner. Then with probability at least $c_{\text{\tiny{\ref{l: submatrix}}}}$ we have
$$\|T\|\geq c_{\text{\tiny{\ref{l: submatrix}}}}\|M\|.$$
\end{lemma}
\begin{proof}
Let $x$ be a unit vector in $\R^n$ which realizes the norm of $M$, i.e.\ $\|Mx\|_2=\|M\|_2$.
Denote by $X$ the {\it random} $\lfloor n/2\rfloor$-dimensional vector defined as
$$X:=\bigl(x_{\sigma(n-\lfloor n/2\rfloor+1)},\ldots,x_{\sigma(n)}\bigr).$$
We will show that $\|TX\|_2\geq c_{\text{\tiny{\ref{l: submatrix}}}}\|M\|$ with probability at least $c_{\text{\tiny{\ref{l: submatrix}}}}$,
for a sufficiently small universal constant $c_{\text{\tiny{\ref{l: submatrix}}}}>0$.

Fix for a moment any $i\leq \lfloor n/2\rfloor$. Note that the inner product of the $i$-th row of $T$
with $X$ can be written as
$$\sum_{j=n-\lfloor n/2\rfloor+1}^n M_{\sigma(i)\sigma(j)}x_{\sigma(j)}.$$
Conditioned on any realization of $\sigma(i)$, the set $\bigl\{\sigma(j)\,:\,j>n-\lfloor n/2\rfloor\bigr\}$
is uniformly distributed in $[n]\setminus \{\sigma(i)\}$, whence,
by Lemma~\ref{l: subsum}, and because of the assumption on the matrix diagonal, we have
\begin{equation}\label{eq: aux 1640}
\Prob\Bigl\{\Bigl|\sum_{j=n-\lfloor n/2\rfloor+1}^n M_{\sigma(i)\sigma(j)}x_{\sigma(j)}\Bigr|\geq
c\Bigl|\sum_{j=1}^n M_{\sigma(i)\sigma(j)}x_{\sigma(j)}\Bigr|\;|\;\sigma(i)\Bigr\}\geq c,
\end{equation}
for some universal constant $c>0$. Next, observe that
\begin{align*}
\|TX\|_2^2&=\sum_{i=1}^{\lfloor n/2\rfloor}\biggl(\sum_{j=n-\lfloor n/2\rfloor+1}^n M_{\sigma(i)\sigma(j)}x_{\sigma(j)}\biggr)^2\\
&\geq c^2\sum_{i=1}^{\lfloor n/2\rfloor}\biggl(\sum_{j=1}^n M_{\sigma(i)\sigma(j)}x_{\sigma(j)}\biggr)^2\chi_i\\
&= c^2\sum_{i=1}^{\lfloor n/2\rfloor}\biggl(\sum_{j=1}^n M_{\sigma(i)j}x_{j}\biggr)^2\chi_i,
\end{align*}
where $\chi_i$ is the indicator of the event
$$\Bigl\{\Bigl|\sum_{j=n-\lfloor n/2\rfloor+1}^n M_{\sigma(i)\sigma(j)}x_{\sigma(j)}\Bigr|\geq
c\Bigl|\sum_{j=1}^n M_{\sigma(i)\sigma(j)}x_{\sigma(j)}\Bigr|\Bigr\}.$$
Therefore $\Vert TX\Vert^2_2\geq c^2\eta$ where $\eta$ is defined as 
$$\eta:=\sum_{i=1}^{\lfloor n/2\rfloor}\biggl(\sum_{j=1}^n M_{\sigma(i)j}x_{j}\biggr)^2\chi_i.$$
In view of \eqref{eq: aux 1640}, we have
$$\E\eta
\geq c\sum_{i=1}^{\lfloor n/2\rfloor}\E\biggl(\sum_{j=1}^n M_{\sigma(i)j}x_{j}\biggr)^2
=\frac{c\lfloor n/2\rfloor}{n}\|Mx\|_2^2=\frac{c\lfloor n/2\rfloor}{n}\|M\|^2.$$
On the other hand, deterministically
$$\eta\leq \sum_{i=1}^{n}\biggl(\sum_{j=1}^n M_{\sigma(i)j}x_{j}\biggr)^2\chi_i\leq\|M\|^2,$$
whence $\E\eta^2\leq \|M\|^2\E\eta$. Applying the Paley--Zygmund inequality to $\eta$, we obtain
$$\Prob\Bigl\{\eta\geq \frac{1}{2}\E\eta\Bigr\}\geq \frac{1}{4}\frac{\bigl(\E\eta\bigr)^2}{\E\eta^2}
\geq \frac{\E\eta}{4\|M\|^2}\geq \frac{c}{12}.$$
Since $\Vert TX\Vert^2_2\geq c^2\eta$ and $\E \eta \geq \frac{c}{3}\Vert M\Vert^2$ then 
$$\|TX\|_2^2\geq \frac{c^3}{6}\|M\|^2$$
with probability at least $\frac{c}{12}$, and the proof is complete.
\end{proof}

Equipped with the above lemma, we are now ready to prove Theorem~\ref{th: main}.

\begin{proof}[Proof of Theorem~\ref{th: main}]
Let $c_{\text{\tiny{\ref{th: main}}}}:=\frac{1}{2}c_{\text{\tiny{\ref{l: submatrix}}}}$,
and let $M$, $\sigma$, $T$ and event $\Event$ be as in the statement of the theorem.
Conditioning on any realization of $M$,
we get, by Lemma~\ref{l: submatrix},
$$\Prob\{\|T\|\geq c_{\text{\tiny{\ref{l: submatrix}}}}\|M\|\;|\;M\}\geq c_{\text{\tiny{\ref{l: submatrix}}}}.$$
Hence,
$$\Prob\{\|T\|\geq c_{\text{\tiny{\ref{l: submatrix}}}}\|M\|\mbox{ AND $\Event$ holds}\;|\;M\}
\geq \frac{c_{\text{\tiny{\ref{l: submatrix}}}}}{2}.$$
Now, fix any $\tau>0$ and denote $\Event_\tau:=\bigl\{\|M\|\geq \tau\bigr\}$.
Using the above relation, we obtain
$$
\Prob\{\|T\|\geq c_{\text{\tiny{\ref{l: submatrix}}}}\tau\mbox{ AND $\Event$ holds}\;|\;\Event_\tau\}\geq
\Prob\{\|T\|\geq c_{\text{\tiny{\ref{l: submatrix}}}}\|M\|\mbox{ AND $\Event$ holds}
\;|\;\Event_\tau\}\geq \frac{c_{\text{\tiny{\ref{l: submatrix}}}}}{2}.$$
It remains to use the definition of conditional probability.
\end{proof}

\section{The second singular value}

To deduce Theorem~\ref{th: main2} from Theorem~\ref{th: main}, we have to identify ``the right'' event for conditioning.
In the following lemma, we show that with high probability, the property of having constant row and column sums
is ``almost true'' for the top right corner of our random matrix provided that it is jointly exchangeable.

\begin{lemma}\label{lem: degree-event}
There exists a universal constant $C_{\text{\tiny{\ref{lem: degree-event}}}}>0$
with the following property. 
Let $A\in \A_n(d)$ be an $n\times n$ random matrix
and let $\delta>0$ be such that
$$C_{\text{\tiny{\ref{lem: degree-event}}}}\|\Row_i(A)\|_2,\,C_{\text{\tiny{\ref{lem: degree-event}}}}\|\Col_i(A)\|_2\leq \delta\;\;\mbox{a.s.}\quad\quad i\leq n.$$
Further, let $\sigma$ be the uniform random permutation on $[n]$ independent from $A$ and let $T_s$ be the
$\lfloor n/2\rfloor\times\lfloor n/2\rfloor$ top right corner of $\sigma(A)$. 
Then, with the event $\Event_{\text{\tiny{\ref{lem: degree-event}}}}$ defined by
\begin{align*}
\Event_{\text{\tiny{\ref{lem: degree-event}}}}
:=\Big\{&\big|\big\{i\leq n/2:\,\big|u_i(T_s)-\frac{d}{2}\big|> k\delta\big\}\big|
\leq n e^{-k^2}\mbox{ for all }k\in\N\;\;\mbox{ AND}\\
&\big|\big\{i\leq n/2:\,\big|v_i(T_s)-\frac{d}{2}\big|> k\delta\big\}\big|
\leq n e^{-k^2}\mbox{ for all }k\in\N\Big\},
\end{align*}
we have 
$$\Prob(\Event_{\text{\tiny{\ref{lem: degree-event}}}}\,|\,A)\geq 1-c_{\text{\tiny{\ref{th: main}}}}.$$
Here, $c_{\text{\tiny{\ref{th: main}}}}$ is the constant from Theorem~\ref{th: main}.
\end{lemma}
\begin{proof}
Set
$$C_{\text{\tiny{\ref{lem: degree-event}}}}:= 4\ln (8/c_{\text{\tiny{\ref{th: main}}}}).$$ 
First, let us condition on any realizaton $\widetilde A$ of the matrix $A$,
so that the randomness comes only from the permutation $\sigma$.
By the definition of $\sigma(A)$ and $T_s$, we have
$$\Row_1(T_s)=\big(a_{\sigma(1)\sigma(n-\lfloor n/2\rfloor+1)},\dots,a_{\sigma(1)\sigma(n)}\big).$$
Since $\sigma$ is uniformly distributed on $[n]$, conditioning on an event $\sigma_1=k$ for any fixed $k\leq n$,
the set $\{\sigma(n-\lfloor n/2\rfloor+1),\dots,\sigma(n)\}$ is uniformly distributed on $[n]\setminus\{k\}$.
Thus, conditioned {\it both} on $A=\widetilde A$ and $\sigma(1)=k$, we have
that $v_1(T_s)=\Vert\Row_1(T_s)\Vert_1$ is equidistributed with the random sum
$$\sum_{j\in \mathcal S}\widetilde a_{kj},$$
where $(\widetilde a_{kj})_{j\leq n}$ is the $k$-th row of $\widetilde A$, and
$\mathcal S$ is a uniform random subset of $[n]\setminus\{k\}$ of cardinality $m:=\lfloor n/2\rfloor$.
Hence, from Lemma~\ref{l: subsum2}, we have for any $t>0$:
\begin{align*}
\Prob\Bigl\{\bigl| v_1(T_s)-\frac{m}{n}\sum_{j\neq k}\widetilde a_{kj}\bigr|>t\;|\;A=\widetilde A\mbox{ AND }\sigma(1)=k\Bigr\}
&\leq 2\exp\bigg(-\frac{2\, t^2}{\|\Row_k(\widetilde A)\|_2^2}\bigg)\\
&\leq 2\exp\Big(-\frac{2{C_{\text{\tiny{\ref{lem: degree-event}}}}}^2\, t^2}{\delta^2}\Big).
\end{align*}
Note that, in view of the choice of $\delta$, we have $a_{kk}\leq \frac{\delta}{C_{\text{\tiny{\ref{lem: degree-event}}}}}$
Thus, we can simply write
\begin{equation}\label{eq: concentration-outdeg}
\Prob\Bigl\{\bigl| v_1(T_s)-\frac{dm}{n}\bigr|>t+\frac{\delta}{C_{\text{\tiny{\ref{lem: degree-event}}}}}\;|\;A=\widetilde A\mbox{ AND }\sigma(1)=k\Bigr\}
\leq 2\exp\Big(-\frac{2{C_{\text{\tiny{\ref{lem: degree-event}}}}}^2\, t^2}{\delta^2}\Big).
\end{equation}
Obviously, analogous estimates hold for all rows of $T_s$.
For any $t>0$, denote by $\eta_t$ the random variable
$$\eta_t:=\Big|\Big\{i\leq n/2:\, \bigl| v_i(T_s)-\frac{dm}{n}\bigr|>\frac{t\, \delta}{2}+\frac{\delta}{C_{\text{\tiny{\ref{lem: degree-event}}}}}\Big\}\Big|.$$
Then, in view of \eqref{eq: concentration-outdeg}, for all $t>0$ we have
$$\E(\eta_t\,|\,A=\widetilde A)
=\sum_{i=1}^{m}\Prob\Big\{\bigl| v_i(T_s)-\frac{dm}{n}\bigr|>\frac{t\, \delta}{2}+\frac{\delta}{C_{\text{\tiny{\ref{lem: degree-event}}}}}\,|\,A=\widetilde A\Big\}
\leq 2m\exp\Big(-\frac{{C_{\text{\tiny{\ref{lem: degree-event}}}}}^2t^2}{2}\Big).
$$
Applying Markov's inequality, we obtain
$$\Prob\big\{\eta_t> me^{-t^2}\,|\,A=\widetilde A\big\}\leq 2\exp\Big(t^2-\frac{{C_{\text{\tiny{\ref{lem: degree-event}}}}}^2t^2}{2}\Big),\quad t>0.$$
Hence, summing up over all natural $t$, we get
$$\Prob\big\{\eta_t> me^{-t^2}\mbox{ for some }t\in\N\,|\,A=\widetilde A\big\}
\leq 2\sum_{t=1}^\infty\exp\Big(t^2-\frac{{C_{\text{\tiny{\ref{lem: degree-event}}}}}^2t^2}{2}\Big)\leq \frac{c_{\text{\tiny{\ref{th: main}}}}}{2},$$
by the choice of $C_{\text{\tiny{\ref{lem: degree-event}}}}$.
In terms of $v(T_s)$, the condition can be rewritten as
$$\Prob\Big\{\big|\big\{i\leq n/2:\,\big|v_i(T_s)-\frac{d}{2}\big|> k\delta\big\}\big|
\leq me^{-k^2}\mbox{ for all }k\in\N\,|\,A=\widetilde A\Big\}
\geq 1-\frac{c_{\text{\tiny{\ref{th: main}}}}}{2}.$$
Obviously, the same probabilistic relation holds for vector $u(T_s)$.
The result follows.
\end{proof}

The next linear algebraic lemma can be viewed as an extension of relation \eqref{eq: s2-courant-fischer}
to the case when a matrix has ``almost constant'' row and column sums.
\begin{lemma}\label{lem: s2-reduction}
Let $m\in\N$ and let $u,v$ be two vectors in $\R^m$ with strictly positive coordinates and $\|u\|_1=\|v\|_1$.
Further, assume that $d,\delta>0$ are such that
$$\|u-d{\bf 1}\|_\infty,\|v-d{\bf 1}\|_\infty\leq \frac{d}{3}$$
and
$$\|u-d{\bf 1}\|_2,\|v-d{\bf 1}\|_2\leq \delta\sqrt{m}.$$
Then for any $A\in\A_m(u,v)$ we have
$$
 \big\Vert A-\frac{d}{m}{\bf 1}\cdot {\bf 1}^t\big\Vert \leq 2\, s_2(A) +6\delta.
$$ 
\end{lemma}
\begin{proof}
Denote by
$D_{u}$ and $D_{v}$ diagonal matrices having 
$u_i$ and $v_i$ ($i\leq m$) as diagonal elements. 
First note that ${D_u}^{\frac12}{\bf 1}$ is an eigenvector of the matrix
$$({D_v}^{-\frac12}A{D_u}^{-\frac12})^t
{D_v}^{-\frac12}A{D_u}^{-\frac12}={D_u}^{-\frac12} A^t {D_v}^{-1}A{D_u}^{-\frac12},$$
with the corresponding eigenvalue equal to $1$,
and, obviously, ${D_u}^{\frac12}{\bf 1}$ has positive coordinates.
Hence, by Lemma~\ref{l: eigenvalue of positive vector},
$$
\lambda_{\max}\big({D_u}^{-\frac12} A^t {D_v}^{-1}A {D_u}^{-\frac12}\big)=1.
$$
Thus, $s_1({D_{v}}^{-\frac12} A{D_{u}}^{-\frac12})=1$, and ${D_{u}}^{\frac12}\, {\bf 1}$ and
${\bf 1}^t {D_{v}}^{\frac12}$
are the associated right and left singular vectors.
Moreover, 
$$\Vert {D_{u}}^{\frac12}\, {\bf 1}\Vert_2^2=\Vert {D_{v}}^{\frac12}\, {\bf 1}\Vert_2^2= \Vert u\Vert_1=\Vert v\Vert_1.$$
Therefore, applying the singular value decomposition, we get
$$
s_2\big({D_{v}}^{-\frac12} A{D_{u}}^{-\frac12}\big) =
\Big\Vert {D_{v}}^{-\frac12} A{D_{u}}^{-\frac12}- \frac{{D_{v}}^{\frac12}\, {\bf 1}\cdot{\bf 1}^t\, {D_{u}}^{\frac12}}
{\Vert u\Vert_1} \Big\Vert.
$$
From the last relation and by the triangle inequality, denoting by $\beta$ the expression
$$\beta:=\Big\Vert \frac{D_{v}\, {\bf 1}\cdot{\bf 1}^t\, D_{u}}
{\Vert u\Vert_1} -\frac{d}{m}{\bf 1}\cdot {\bf 1}^t \Big\Vert,$$
we obtain
\begin{align}\label{eq1: s2-reduction}
\big\Vert A-\frac{d}{m}{\bf 1}\cdot {\bf 1}^t\big\Vert
&\leq \Big\Vert A- \frac{D_{v}\, {\bf 1}\cdot{\bf 1}^t\, D_{u}}
{\Vert u\Vert_1} \Big\Vert +\beta\nonumber\\
&\leq \Vert {D_{v}}^{\frac12}\Vert\, \Vert {D_{u}}^{\frac12}\Vert\, s_2\big({D_{v}}^{-\frac12} A{D_{u}}^{-\frac12}\big) 
+ \beta\nonumber\\
&\leq \Vert {D_{v}}^{\frac12}\Vert\, \Vert {D_{u}}^{\frac12}\Vert\, \Vert {D_{v}}^{-\frac12}\Vert\, \Vert {D_{u}}^{-\frac12}\Vert\, 
s_2(A)+ \beta,
\end{align}
where in the last step we used the relation
$s_2\big(D_{v}^{-\frac12} AD_{u}^{-\frac12}\big)\leq  \Vert D_{v}^{-\frac12}\Vert\, \Vert D_{u}^{-\frac12}\Vert\, s_2(A)$. 
By the assumptions on vectors $u,v$ and on $d$, we have
\begin{equation}\label{eq: bounds-in-d}
\frac{2d}{3}\leq \min\{u_i,v_i,\;i\leq m\}\leq \max\{u_i,v_i,\;i\leq m\}\leq \frac{4d}{3}.
\end{equation}
Hence, by \eqref{eq1: s2-reduction}, we get 
\begin{equation}\label{eq2: s2-reduction} 
\big\Vert A-\frac{d}{m}{\bf 1}\cdot {\bf 1}^t\big\Vert\leq 2\, s_2(A)+  \Big\Vert \frac{D_{v}\, {\bf 1}\cdot{\bf 1}^t\, D_{u}}
{\Vert u\Vert_1} -\frac{d}{m}{\bf 1}\cdot {\bf 1}^t \Big\Vert.
\end{equation}
It remains to estimate $\beta$.
By the triangle inequality,
\begin{align*}
\beta&\leq  \Big\Vert \frac{(D_{v}-d\, I)\, {\bf 1}\cdot{\bf 1}^t\, D_{u}}
{\Vert u\Vert_1}\Big\Vert + \Big\Vert \frac{d\, {\bf 1}\cdot{\bf 1}^t\, (D_{u}-d\, I)}
{\Vert u\Vert_1}\Big\Vert +\Big\Vert \frac{d^2 }{\Vert u\Vert_1}{\bf 1}\cdot {\bf 1}^t  -\frac{d}{m}{\bf 1}\cdot {\bf 1}^t \Big\Vert\\
&= \frac{\Vert v-d{\bf 1}\Vert_2\, \Vert u \Vert_2}{\Vert u\Vert_1}+ 
\frac{d\sqrt{m}\, \Vert u-d{\bf 1}\Vert_2}{\Vert u\Vert_1} +d\, \frac{\big\vert md-\Vert u\Vert_1\big\vert}{\Vert u\Vert_1},
\end{align*}
where we used the identity $\Vert yz^t\Vert = \Vert y\Vert_2\, \Vert z\Vert_2$ which holds for any two vectors $y,z\in\R^m$. 
By the assumptions on $u$ and $v$ and a standard relation between $\|\cdot\|_2$ and $\|\cdot\|_1$-norms, we have
$$
\Vert u-d{\bf 1}\Vert_2, \Vert v-d{\bf 1}\Vert_2\leq \delta\, \sqrt{m}\quad \text{and}\quad 
\Vert u-d{\bf 1}\Vert_1\leq \delta\, m.
$$
This, together with \eqref{eq: bounds-in-d}, implies that 
$
\beta\leq 6 \delta.
$
Combining this relation with \eqref{eq2: s2-reduction}, we finish the proof.
\end{proof}

\begin{proof}[Proof of Theorem~\ref{th: main2}]
Let $A\in \A_n(d)$ be an $n\times n$ jointly exchangeable random matrix
and let $\delta>0$ be such that
$$C_{\text{\tiny{\ref{lem: degree-event}}}}\|\Row_i(A)\|_2,\,C_{\text{\tiny{\ref{lem: degree-event}}}}\|\Col_i(A)\|_2\leq \delta\;\;\mbox{a.s.}\quad\quad i\leq n,$$
where the constant $C_{\text{\tiny{\ref{lem: degree-event}}}}$ comes from Lemma~\ref{lem: degree-event},
and such that
\begin{equation}\label{eq: aux 124}
d/\sqrt{\ln n}\geq 100\delta.
\end{equation}
Let $\sigma$ be a uniform random permutation on $[n]$ independent from $A$.
In view of the definition of the joint exchangeability, $\sigma(A)$
is equidistributed with $A$. In particular, denoting by $T_s$ the $\lfloor n/2\rfloor\times\lfloor n/2\rfloor$ top right corner of $\sigma(A)$,
and by $T$ --- the $\lfloor n/2\rfloor\times\lfloor n/2\rfloor$ top right corner of $A$,
we get that $T$ and $T_s$ are equidistributed.
Moreover, we clearly have $s_2(A)=s_2(\sigma(A))$ deterministically.
Thus, to prove the theorem, it is enough to show that for all $L\geq C$ (for a sufficiently large constant $C$) we have
$$\Prob\big\{s_2(A)\geq L\delta\big\}\leq \frac{1}{c}\Prob\Big\{ s_2(T_s)\geq cL\delta \mbox{ \ AND }
\big(u(T_s),v(T_s)\big)\in \degreeset_{\lfloor n/2\rfloor}\big(d/2,\delta\big)\Big\}.
$$
By \eqref{eq: s2-courant-fischer}, we have
$
s_2(A)= \Vert A-\frac{d}{n}{\bf 1}\cdot{\bf 1}^t\Vert.
$
It is easy to see that the matrix
${\rm Diag}(A-\frac{d}{n}{\bf 1}\cdot{\bf 1}^t)$ has norm at most $\max\limits_{i\leq n}\|\Row_i(A)\|_2$, whence 
\begin{equation}\label{eq: s2-triangle-ineq}
s_2(A)\leq \max\limits_{i\leq n}\|\Row_i(A)\|_2+\Vert B\Vert\leq \delta +\Vert B\Vert, 
\end{equation}
where $B:= A-\frac{d}{n}{\bf 1}\cdot{\bf 1}^t- {\rm Diag}\big(A-\frac{d}{n}{\bf 1}\cdot{\bf 1}^t\big)$. 

Note that the $\lfloor n/2\rfloor\times\lfloor n/2\rfloor$ top right corner of $\sigma( B)$ 
is equal to $T_s -\frac{d}{n}{\bf 1}\cdot{\bf 1}^t$, where the vector ${\bf 1}$
is assumed to have dimension $\lfloor n/2\rfloor$.
Applying Theorem~\ref{th: main} to matrix $B$ and event $\Event_{\text{\tiny{\ref{lem: degree-event}}}}$
(taken from Lemma~\ref{lem: degree-event}), we get 
$$
f(\tau):=\Prob\bigl\{\|B\|\geq \tau\bigr\}
\leq \frac{1}{c_{\text{\tiny{\ref{th: main}}}}}
\Prob\Bigl\{\big\|T_s-\frac{d}{n}{\bf 1}\cdot{\bf 1}^t\big\|\geq c_{\text{\tiny{\ref{th: main}}}}\tau\mbox{ AND }
\Event_{\text{\tiny{\ref{lem: degree-event}}}}\mbox{ holds}\Bigr\}$$
for all $\tau>0$.
Note that, in view of \eqref{eq: aux 124} and the definition of $\Event_{\text{\tiny{\ref{lem: degree-event}}}}$, we have
$$\Big\{\big\|u(T_s)-\frac{d}{2}{\bf 1}\big\|_\infty,\,\big\|v(T_s)-\frac{d}{2}{\bf 1}\big\|_\infty\leq\frac{d}{6}\}\supset 
\Event_{\text{\tiny{\ref{lem: degree-event}}}}.$$
Further, it is easy to see that
$$\Big\{\big\|u(T_s)-\frac{d}{2}{\bf 1}\big\|_2,\,\big\|v(T_s)-\frac{d}{2}{\bf 1}\big\|_2\leq 4\delta\sqrt{n}\Big\}\supset 
\Event_{\text{\tiny{\ref{lem: degree-event}}}}.$$
Hence, by Lemma~\ref{lem: s2-reduction}, we get 
$
f(\tau)
\leq  \frac{1}{c_{\text{\tiny{\ref{th: main}}}}}\Prob\{s_2(T_s)\geq \frac{c_{\text{\tiny{\ref{th: main}}}}}{2}\tau-24\delta \mbox{ AND }
\Event_{\text{\tiny{\ref{lem: degree-event}}}}\mbox{ holds}\}.
$
Together with \eqref{eq: s2-triangle-ineq}, it gives
$$
\Prob\bigl\{s_2(A)\geq \delta+\tau\bigr\}\leq \frac{1}{c_{\text{\tiny{\ref{th: main}}}}}
\Prob\Big\{s_2(T_s)\geq \frac{c_{\text{\tiny{\ref{th: main}}}}}{2}\tau-24\delta \mbox{ AND }
\Event_{\text{\tiny{\ref{lem: degree-event}}}}\mbox{ holds}\Big\},\quad \tau>0.
$$
The result follows.
\end{proof}

\bigskip

{\bf Acknowledgments.}
The first named author is partially supported by the Simons Foundation ({\it{}Collaboration on Algorithms and Geometry}).

\bigskip

\noindent {\small Konstantin Tikhomirov,}\\
{\small Department of Mathematics, Princeton University,}\\
{\small E-mail: kt12@math.princeton.edu}

\bigskip

\noindent {\small Pierre Youssef,}\\
{\small Laboratoire de Probabilit\'es et de Mod\`eles al\'eatoires,
Universit\'e Paris Diderot,}\\
{\small E-mail: youssef@math.univ-paris-diderot.fr}

\end{document}